\documentclass[a4paper,reqno]{amsart}

\usepackage{graphicx}
\usepackage{mathrsfs}
\usepackage{amsfonts}
\usepackage{amssymb}
\usepackage{amsmath}
\usepackage{amsthm}
\usepackage{amscd}
\usepackage[all,2cell]{xy}
\usepackage{color}
\usepackage[pagebackref,colorlinks]{hyperref}
\usepackage{enumerate}
\usepackage{bm}
\usepackage{bbm}
\usepackage[normalem]{ulem}
\usepackage{mathrsfs}

\usepackage{comment}
\usepackage{hyperref}
\UseAllTwocells \SilentMatrices

\numberwithin{equation}{section}

\usepackage{comment}

\theoremstyle{plain}
\newtheorem{thm}{Theorem}[section]

\newtheorem{cor}[thm]{Corollary}

\theoremstyle{definition}
\newtheorem{deff}[thm]{Definition}

\theoremstyle{remark}
\newtheorem{rmk}[thm]{\bf Remark}

\def\mG{\mathcal{G}}
\def\xra{\xrightarrow[]{}}
\def\a{\alpha}
\def\b{\beta}

\newcommand{\Isoo}{\operatorname{Iso}}
\newcommand{\spann}{\operatorname{span}}

\def\-{\text{-}}

\newcommand{\KP}{\operatorname{KP}}

\begin{document}

\title{A note on the centraliser of a subalgebra of Steinberg algebra}

\author{Roozbeh Hazrat}
\address{
Centre for Research in Mathematics and Data Sceince\\
Western Sydney University\\
Australia} \email{r.hazrat@westernsydney.edu.au}

\author{Huanhuan Li}

\address{ 1 School of Mathematical Sciences, Anhui University, Hefei, Anhui Province, China}
\address{
2 Centre for Research in Mathematics and Data Sceince\\
Western Sydney University\\
Australia} \email{h.li@westernsydney.edu.au}



\date{\today}

\begin{abstract} 
For an ample Hausdorff groupoid $\mG$, and the Steinberg algebra $A_R(\mG)$ with coefficients in the commutative ring $R$ with unit,  we describe the centraliser of subalgebra $A_R(U)$ with $U$ an open closed invariant subset of unit space of $\mG$. In particular, we obtain that the algebra of  the interior of the isotropy is indeed the centraliser of the diagonal subalgebra of Steinberg algebra.  This will unify several results in the literature and  the corresponding results for Leavitt path algebras follow. 
\end{abstract}

\maketitle

\section{Introduction}

Let $E$ be a graph and $L_R(E)$  the Leavitt path algebra associated to $E$ with coefficients in the commutative ring $R$. The $R$-subalgebra of $L_R(E)$ generated by all the monomials $aa^*$, where $a$ is a path in $E$, is called the 
\emph{diagonal subalgebra} and denoted by $D_R(E)$. The \emph{commutative core} subalgebra  of $L_R(E)$ is $R$-subalgebra generated by all the monomials $aba^*$, where $a$ is a path and $b$ is a cycle with no exit and denoted by $M_R(E)$.  The commutative core subalgebra of a Leavitt path algebras were considered by Gil Cantoa and Nasr-Isfahani \cite{gil}. They prove that $M_R(E)$ is a maximal commutative subalgebra of $L_R(E)$ which coincides with  the centraliser of $D_R(E)$. They also raised an open question as to which $R$ and $E$, the Leavitt path algebra $L_R(E)$ has the property that $Z(L_R(E)) = M_R(E)$, where $Z(L_R(E))$ is the centre of $L_R(E)$. Their results are the algebraic analogue of an earlier work by Nagy and Reznikoff who introduced the $M_R(E)$ in the setting of graph $C^*$-algebras and proved that it is a masa~\cite{nagy}. A key tool in the above papers is the conditional expectation which was used in \cite{nagy} and then developed in \cite{gil} in the algebraic setting. The diagonal subalgebra $D_R(E)$ and consequently the core subalgebra $M_R(E)$ play vital roles in the structure and classification theory of graph algebras~\cite{carl}. 

In this note we consider the centraliser of a subalgebra of Steinberg algebra. In particular, we show that the core is indeed the centraliser of the diagonal subalgebra of Steinberg algebra. Despite the fact that these algebras are much more general than graph algebras, the proofs we present here are very short. Our approach does not require the use of conditional expectation and therefore it is direct and short.  Specialising to Leavitt path algebras we obtain the result of \cite{gil} and answer the open question raised there. 

\section{The centraliser of a subalgebra of Steinberg algebra}

We first recall the notions of ample groupoids and Steinberg algebras. We describe the centraliser of a subalgebra  associated to an open invariant subset of the unit space of the given groupoid. In particular, we show that the algebra of  the interior of the isotropy is indeed the centraliser of the diagonal subalgebra of Steinberg algebra; and we translate this to the settings of Leavitt path algebras and Kumjian--Pask algebras.

\subsection{Ample groupoids and Steinberg algebras}

A groupoid is a small category in which every morphism is invertible. For a groupoid $\mG$ and  $g\in\mG$,  denote $s(g):=g^{-1}g$ and $r(g):=gg^{-1}$. The pair $(g_1,g_2)$ is composable if and only if $r(g_2)=s(g_1)$. The set
$\mG^{(0)}:=s(\mG)=r(\mG)$ is called the \emph{unit space} of $\mG$.  The \emph{isotropy group} at a unit $u$ of $\mG$ is the group ${\rm Iso}(u) =\{g\in\mG \mid s(g)=r(g)=u\}$. Let ${\rm Iso}(\mG)=\bigsqcup_{u\in\mG^{(0)}}{\rm Iso}(u)$. For
$U,V\in\mG$, we define
\[
    UV=\big \{g_1 g_2 \mid g_1\in U, g_2\in V \text{ and } r(g_2)=s(g_1)\big\}.
\]

A \emph{topological groupoid} is a groupoid endowed with a topology under which the inverse map
is continuous, and such that composition is continuous with respect to the relative
topology on $\mG^{(2)} := \{(g_1,g_2) \in \mG \times \mG : s(g_1) = r(g_2)\}$ inherited from
$\mG\times \mG$. An \emph{\'etale} groupoid is a topological groupoid $\mG$ such that the
domain map $s:\mG \rightarrow \mG^{(0)}$ is a local homeomorphism. In this case, the range map $r$ is also a local
homeomorphism. Throughout this article, we assume that the \'etale groupoids are locally compact Hausdorff. 
An \emph{open bisection} of $\mG$ is an open subset $U\subseteq \mG$ such
that $s|_{U}$ and $r|_{U}$ are homeomorphisms onto an open subset of $\mG^{(0)}$. If $\mG$ is an \'etale groupoid, then there is a
base for the topology on $\mG$ consisting of open bisections with compact closure. As demonstrated in \cite{st}, if $\mG^{(0)}$ is totally disconnected and $\mG$ is \'etale, then there is a basis for the topology on $\mG$ consisting of compact open bisections. We say
that an \'etale groupoid $\mG$ is \emph{ample} if there is a basis consisting of compact open bisections for its topology. For an ample groupoid $\mG$ and a commutative ring $R$ with the discrete topology,  the $R$-value maps on $\mG$ which are locally constant and have compact support constitute an algebra called \emph{Steinberg algebra} and denoted by $A_R(\mG)$. One can show that \cite[\S 4]{ch}
\begin{align*}
A_R(\mG )& =  \Big \{\sum_{B\in F} r_B 1_B \mid  F: \text{ mutually disjoint finite collection of compact open bisections} \Big \}, \\
& \text{ addition and scalar multiplication of functions are pointwise},\\
& \text{ multiplications on the generators are }  1_B 1_D = 1_{BD}.
\end{align*}

Throughout this note we work with the Steinberg algebras associated to the ample Hausdorff groupoid $\mG$, its  unit space $\mG^{(0)}$ and its the interior of isotropy  $\Isoo(\mG)^\circ$.  

\subsection{The centraliser of a subalgebra of Steinberg algebra}

Let $\mathcal{G}$ be an ample Hausdorff groupoid. 
One can consider $A_R(\mG^{(0)})$ and $A_R(\Isoo(\mG)^\circ)$ as subalgebras of $A_R(\mG)$.  For basics of the theory of Steinberg algebras we refer the reader to \cite{ch,st}.

\begin{deff}
Let $\mG$ be an ample Hausdorff groupoid and $A_R(\mG)$ the Steinberg algebra associated to $\mG$. We call the subalgebra $A_R(\Isoo(\mG)^\circ)$ the \emph{core} algebra of $A_R(\mG)$.
\end{deff}

We need the following notation: For a subset $X$ of a ring $A$, its \emph{centraliser} (or \emph{commutant})  is defined as $C_A(X):= \{a \in A \mid xa=ax, \text{ for all }  x \in X \}$. For a subring $B\subseteq A$, we have $C_A(B)=B$ if and only if $B$ is a maximal commutative subring of $A$. 

\begin{thm}\label{maniiik}
Let $\mG$ be an ample Hausdorff groupoid and $A_R(\mG)$ the Steinberg algebra associated to $\mG$. Let $U$ be an open closed invariant subset of $ \mG^{(0)}$. Then
\begin{equation}\label{coreyy}
C_{A_R(\mG)}\big (A_R(U)\big) = A_R(\Isoo(\mG)^\circ)+A_R(\mG_{\mG^{(0)}\setminus U}).
\end{equation}
\end{thm}
\begin{proof}
Let $V \subseteq \Isoo(\mG)^\circ$ be a compact open bisection and $W \subseteq \mG^{(0)}$ a compact open set. We first observe that $WV=VW$. If $x \in WV$, then $x=wv$, where $w\in W \subseteq \mG^{(0)} $ and $v\in V$. Since $s(v)=r(v)=w$, it follows that $x=wvw=vw \in VW$. The converse follows from symmetry.  
Now for any element $x\in A_R(\Isoo(\mG)^\circ)$ and $y\in  A_R(U)$, write 
$x=\sum r_i 1_{V_i}$ and $y=\sum s_j 1_{W_j}$, where $V_i \subset \Isoo(\mG)^\circ$ and $W_j \subseteq U$ are compact open bisections. Since $V_i$ and $W_j$ commutes, we have $xy=yx$. This shows that  $A_R(\Isoo(\mG)^\circ) \subseteq C_{A_R(\mG)}\big (A_R(U)\big)  .$  Clearly that $A_R(\mG_{\mG^{(0)}\setminus U})\subseteq C_{A_R(\mG)}(A_R(U))$ as for any $f\in A_R(\mG_{\mG^{(0)}\setminus U})$ and $g\in A_R(U)$, we have $fg=0=gf$. 

Now suppose $f\in A_R(\mG)$ commutes with elements of $A_R(U)$. Write $f=\sum_{i=1}^n r_i 1_{V_i}$, where $0\not = r_i \in R$ and $V_i \subset \mG$  are disjoint compact open bisections. 
For each $V_i$, we have that $V_i^1=d^{-1}(d(V_i)\cap U)\cap V_i$ and $V_i^2=d^{-1}(d(V_i)\cap U^c)\cap V_i $ are compact with $U^c$ the completion of $U$. It follows that $1_{V_i}=1_{V_i^1}+1_{V_i^2}$. Observe that 
$1_{V_i^2}\in A_R(\mG_{\mG^{(0)}\setminus U})$. It suffices to show that for each $i$ we have $1_{V_i^1}\in A_R(\Isoo(\mG)^{o})$. That is, we need to show that $V_i^1\subseteq \Isoo(\mG)$, for all $i$'s. Suppose that there is $x \in V_k^1$ such that $s(x)\not = r(x)$. Since $\mG^{(0)}$ is Hausdorff, there is a compact open set $W \subseteq U \subseteq \mG^{(0)}$ such that $s(x)\in W$ but $r(x) \not \in W$. It follows that $x\in V_k^1W$ but $x\not \in WV_k^1$. Since $\sum_{i=1}^n r_i1_{V_i^1}\in C_{A_R(\mG)}(A_R(U))$, we have $\sum_{i=1}^n r_i1_{V_i^1}1_W=1_W \sum_{i=1}^n r_i1_{V_i^1}$. It follows that 
\begin{equation}\label{hghyty}
\Big( \sum_{i=1}^n r_i 1_{WV_i^1}\Big)(x) =\Big(\sum_{i=1}^n r_i 1_{V_i^1 W}\Big)(x).
\end{equation}
Note that $WV_1^1, \dots WV_n^1$ are disjoint and so are $V_1^1W,\dots V_n^1W$. Thus (\ref{hghyty}) reduces to $r_k=0$ which is a contradiction. So $s(x)=r(x)$. This implies that $V_i^1 \subseteq \Isoo(\mG)$ and consequently $1_{V_i^1} \in A_R(\Isoo(\mG)^\circ)$. 
\end{proof}

By Theorem \ref{maniiik} when $U=\mathcal{G}^{(0)}$, we have $C_{A_R(\mathcal{G})}(A_R(\mathcal{G}^{(0)}))=A_R(\Isoo(\mathcal{G})^o)$ which was given by \cite[Proposition 2.2]{st2}.


For the following consequence we assume that  the interior of the isotropy $\Isoo(\mG)^\circ$, is abelian (that is, it is a bundle of abelian groups). A large class of topological groupoids have abelian isotropy, such as Deaconu-Renault groupoids~\cite{rr} (and in particular graph groupoids \S\ref{hfyhf}). 

\begin{cor}\label{colkfhgy}
Let $\mG$ be an ample Hausdorff groupoid and $A_R(\mG)$ the Steinberg algebra associated to $\mG$. Suppose $\Isoo(\mG)^\circ$ is abelian. Then we have the followings. 

\begin{enumerate}[\upshape(i)]

\item The algebra $A_R(\Isoo(\mG)^\circ)$ is a maximal commutative subalgebra of $A_R(\mG)$.

\medskip 

\item The Steinberg algebra $A_R(\mG)$ is commutative if and only if $Z(A_R(\mG))=C_{A_R(\mG)}\big (A_R(\mG^{(0)}\big)$.

\end{enumerate}
\end{cor}
\begin{proof}

(1) Since $\Isoo(\mG)^\circ$ is abelian, the subalgebra $A_R(\Isoo(\mG)^\circ)$ is commutative. Now the results follows from Theorem~\ref{maniiik}.

(2) Suppose  $Z(A_R(\mG))=C_{A_R(\mG)}\big (A_R(\mG^{(0)}\big)$. Let $V$ be a compact open bisection of $\mG$. Suppose that there is $x \in V$ such that $s(x)\not = r(x)$. Since $\mG^{(0)}$ is Hausdorff, there is a compact open set $U \subseteq \mG^{(0)}$ such that $s(x)\in U$ but $r(x) \not \in U$. It follows that $x\in VU$ but $x\not \in UV$. But by assumption, $1_U \in A_R(\mG^{(0)}) \subseteq Z(A_R(\mG))$ and thus $1_{UV}=1_U 1_V = 1_V 1_U= 1_{VU}$, i.e., $UV=VU$ which is a contradiction. Thus $V \subseteq \Isoo(\mG)$.  So $A_R(\mG)=A_R(\Isoo (\mG)^{o})$, and therefore $A_R(\mG)$ is commutative.  The converse is clear. 
\end{proof}

If $\mG$ is effective, then $\mG^{(0)}=\Isoo(\mG)^\circ$, therefore by Corollary~\ref{colkfhgy}(i), $A_R(\mG^{(0)})$ is a maximal commutative subalgebra of $A_R(\mG)$. Since topologically principal is a stronger condition than effectiveness, we obtain~\cite[Lemma 2.1]{bosa} from Corollary~\ref{colkfhgy}

\subsection{Commutative cores of Leavitt path algebras and Kumjian--Pask algebras}\label{hfyhf}
Let $E=(E^0,E^1,s,r)$ be a directed graph. We refer the reader to \cite{lpabook} for the basics on the theory of Leavitt path algebras. We denote by $E^{\infty}$ the set of
infinite paths in $E$. Set
$
X := E^{\infty}\cup  \{\mu\in E^{*}   \mid   r(\mu) \text{ is not a regular vertex}\}.
$
Let
\[
\mG_{E} := \big \{(\a x,|\a|-|\b|, \b x) \mid   \a, \b\in E^{*}, x\in X, r(\a)=r(\b)=s(x)\big \}.
\]
We view each $(x, k, y) \in \mG_{E}$ as a morphism with range $x$ and source $y$. The
formulas $(x,k,y)(y,l,z)= (x,k + l,z)$ and $(x,k,y)^{-1}= (y,-k,x)$ define composition
and inverse maps on $\mG_{E}$ making it a groupoid with $\mG_{E}^{(0)}=\{(x, 0, x) \mid
x\in X\}$ which we identify with the set $X$.
Next, we describe a topology on $\mG_{E}$. For $\mu\in E^{*}$ define
\[
Z(\mu)= \{\mu x \mid x \in X, r(\mu)=s(x)\}\subseteq X.
\]
For $\mu\in E^{*}$ and a finite $F\subseteq s^{-1}(r(\mu))$, define
\[
Z(\mu\setminus F) = Z(\mu) \setminus \bigcup_{\a\in F} Z(\mu \a).
\]
The sets $Z(\mu\setminus F)$ constitute a basis of compact open sets for a locally
compact Hausdorff topology on $X=\mG_{E}^{(0)}$ (see \cite[Theorem 2.1]{we}).
For $\mu,\nu\in E^{*}$ with $r(\mu)=r(\nu)$, and for a finite $F\subseteq E^{*}$ such
that $r(\mu)=s(\a)$ for $\a\in F$, we define
\[
Z(\mu, \nu)=\{(\mu x, |\mu|-|\nu|, \nu x) \mid  x\in X, r(\mu)=s(x)\},
\]
and then
\[
Z((\mu, \nu)\setminus F) = Z(\mu, \nu) \setminus \bigcup_{\a\in F}Z(\mu\a, \nu\a).
\]
The sets $Z((\mu, \nu)\setminus F)$ constitute a basis of compact open bisections for a
topology under which $\mG_{E}$ is an ample groupoid. We have an isomorphism 
the map
\begin{align*}
\pi_{E} : L_{R}(E) &\longrightarrow A_{R}(\mG_{E})\\
v &\longmapsto 1_{Z(v)},\\
e &\longmapsto 1_{Z(e, r(e))},\\
e^* &\longmapsto 1_{Z(r(e), e)},
\end{align*} 
where $v\in E^{0}$ and $e\in E^{1}$. Recall that the diagonal subalgebra $D_R(E)$ is generated by monomials $aa^*$, where $a$ is a path in $E$, whereas the commutative core subalgebra  $M_R(E)$ is generated the monomials $aba^*$, where $a$ is a path and $b$ is a cycle with no exit. 
Observe that the isomorphism $\pi_E$ restricts to isomorphisms 
\begin{align}\label{assteinberg}
D_R(E) &\longrightarrow A_R(\mG^{(0)}) \\
M_R(E) &\longrightarrow  A_R(\Isoo(\mG)^{\circ}). \notag
\end{align}

We can recover the results of \cite{gil} and answer the open question raised there \cite[p. 245]{gil}. 

\begin{cor}
Let $E$ be an arbitrary graph and $L_R(E)$ the Leavitt path algebra associated to $E$. Then 

\begin{enumerate}[\upshape(i)]

\item The centraliser of the diagonal algebra $D_R(E)$ is the core algebra $M_R(E)$. 

\item The core algebra $M_R(E)$ is the maximal commutative subalgebra of $L_R(E)$. 

\item If $Z(L_R(E))= M_R(E)$ then $L_R(E)$ is either $R$ or $R[x,x^{-1}]$, i.e, the graph $E$ is either a single vertex or a vertex with a loop. 

\end{enumerate}
\end{cor}
\begin{proof}

(i) and (ii) immediately follow follow from Theorem~\ref{maniiik} and Corollary~\ref{colkfhgy} by considering the graph groupoid $\mG_E$ (see \ref{assteinberg}). 

(iii) By Corollary~\ref{colkfhgy}, $L_R(E)$ is commutative. Thus the graph $E$ is either a single vertex or a vertex with a loop. 
\end{proof}

In a similar manner we can determine the centraliser of the diagonal algebra of Kumjian-Pask algebras as they can also be described as a Deaconu-Renault groupoid algebras. 
Let $\Lambda$ be a row-finite $k$-graph without sources and $\KP_K(\Lambda)$ the
Kumjian--Pask algebra of $\Lambda$. We refer the reader to \cite{pchr} for the basics on the theory Kumjian-Pask algebras.  Following \cite{pchr}, $\Lambda^\infty$ denotes the set of all degree-preserving functor $x : \Omega_{k}\xra \Lambda$. Here  $\Omega_k$ is the $k$-graph defined as a set by $\Omega_k = \{(m,n)
\in \mathbb{N}^k \times \mathbb{N}^k : m \le n\}$ with $d(m,n) = n-m$, $\Omega_k^0 =
\mathbb{N}^k$, $r(m,n) = m$, $s(m,n) = n$ and $(m,n)(n,p) = (m,p)$.

The \emph{diagonal algebra} $D_R(\Lambda)$ is generated by $\lambda \lambda^*$ where $\lambda \in \Lambda$. We define the \emph{core} subalgebra 
\[M_R(\Lambda):=\spann_R \, \big \langle \lambda \mu^* \mid  \lambda x = \mu x, \text{ for any } x\in \Lambda^{\infty} \big\rangle.\]

By appealing to  Theorem~\ref{maniiik} and Corollary~\ref{colkfhgy} , similar to the case of Leavitt path algebras, we can recover the results of \cite[Theorem 4.6]{ccn}. 
\begin{cor}\label{kpk}
Let $\Lambda$ be a row-finite $k$-graph with no sources and  $\KP_K(\Lambda)$ the
Kumjian--Pask algebra associated to $\Lambda$. Then 

\begin{enumerate}[\upshape(i)]

\item The centraliser of the diagonal algebra $D_R(\Lambda)$ is the core algebra $M_R(\Lambda)$. 

\item The core algebra $M_R(\Lambda)$ is the maximal commutative subalgebra of $\KP_R(\Lambda)$. 

\item If $Z(\KP_R(\Lambda))= M_R(\Lambda)$ then $\KP_R(\Lambda)$ is $R[x_1^{\pm 1},  \dots, x_k^{\pm 1}]$. 

\end{enumerate}
\end{cor}

\begin{rmk}
In the setting of groupoid $C^*$-algebras, the general version of Theorem~\ref{maniiik} yet to be established. In this setting, with an extra assumption that $\Isoo(\mG)^\circ$ is abelian, the $C^*$-version of Theorem~\ref{maniiik} was established in \cite{toke} (see \cite[Corollary 5.3, 5.4]{toke}), and also \cite[Theorem~4.3]{brown}). The latter Theorem was used to give the $C^*$-version of Corollary~\ref{kpk} (i) and (ii) (see \cite[Corollary 4.6]{brown}). 
We finish the note by remarking that Exel, Clark and Pardo~\cite{cep} have established the algebraic version of uniqueness theorem of \cite[Theorem 3.1]{brown}. Namely, let $\mG$ be a second-countable, ample, Hausdorff groupoid, $R$ a unital commutative ring, and let $\pi: A_R(\mG) \rightarrow A$ be a homomorphism of rings.  Then $\pi$ is injective if and only if the restriction of $\pi$ on the core subalgebra $A_R(\Isoo(\mG)^{\circ})$ is injective. 

\end{rmk}

\section{Acknowledgements} The authors would like to acknowledge Australian Research Council grant DP160101481. They would like to thank Aidan Sims for fruitful discussions and for pointing them to the papers \cite{brown,toke}.

\end{document}